\documentclass{CSML} %[11pt,leqno,a4paper]{lmcs}
\pdfoutput=1 % force pdf output

\newcommand{\dOi}{13(2:3)2017}
\lmcsheading{}{1--13}{}{}{Oct.~27, 2016}{May~08, 2017}{}

\usepackage{latexsym,amsmath,amssymb,amsthm}
\usepackage[all]{xy}
\usepackage{enumitem}
\usepackage{color} %to make colorized comments in the text
\usepackage{hyperref}

\hypersetup{
  colorlinks=true,
  linkcolor=black,
  citecolor=black,
  filecolor=black,
  urlcolor=black,
}

\newcommand{\ot}{\otimes}

\newcommand{\N}{\mathbb{N}}

\newcommand{\BV}{\mathcal{V}}

\newcommand{\Mod}{\mathbf{Mod}}

\newcommand{\Comon}{\mathsf{Comon}}
\newcommand{\Bimon}{\mathsf{Bimon}}

\newcommand{\op}{\mathsf{op}}
\newcommand{\Cop}{\mathsf{cop}}

\newcommand{\ma}{\mathsf{A}}
\newcommand{\mb}{\mathsf{B}}
\newcommand{\mc}{\mathsf{C}}

\newcommand{\mh}{\mathsf{H}}

\newcommand{\ml}{\mathsf{L}}

\newcommand{\Set}{\mathbf{Set}}

\newcommand{\Alg}{\mathbf{Alg}}
\newcommand{\Ab}{\mathsf{Ab}}

\newcommand{\Mon}{\mathsf{Mon}}

\newcommand{\Bialg}{\mathbf{Bialg}}

\newcommand{\Hopf}{\mathsf{Hopf}}
\newcommand{\Lie}{\mathsf{Lie}}

\newcommand{\lra}{\rightarrow}

\newcommand{\xra}{\xrightarrow}
\newcommand{\im}{\mathrm{Im}\,}

\begin{document}

\title{Hopf and Lie algebras in  semi-additive Varieties}

\author{Hans-E. Porst}
\thanks{Permanent address: Department of Mathematics, University of Bremen, 28359 Bremen, Germany.} 

\address{Department of Mathematical Sciences, University of Stellenbosch,  Stellenbosch, South Africa. }
\email {porst@uni-bremen.de} 

\dedicatory{To  Ji\v r\'i Ad\'amek for his 70 $\!^{th}\!$ birthday}

\begin{abstract}
We study Hopf  and Lie algebras in entropic semi-additive varieties 
%(equivalently, entropic J\'onsson-Tarski varieties and categories of semimodules over a commutative semiring) 
with an emphasis on adjunctions related to the enveloping monoid functor and the primitive element functor. These investigations are in part based on the concept of the abelian core of a semi-additive variety  and its monoidal structure in case the variety is entropic. 
%\vskip 6pt
\end{abstract}
\maketitle

\section*{Introduction}
 The class of entropic varieties, that is, varieties whose algebraic theory is commutative, provides the canonical setting for generalizing classical Hopf algebra theory. By this the following is meant: In every entropic variety one not only can (as in any symmetric monoidal category) express the fundamental concepts of Hopf algebra theory, but also prove a substantial part of  the theory of Hopf algebras.
  This is shown in \cite{Por}, a paper based on  combining results and methods from the theories of varieties,  locally presentable categories,  and  coalgebra, of which  I had the pleasure of learning a lot in my collaboration with Ji\v r\'i Ad\'amek (see e.g. \cite{AP0}, \cite{AP1}and \cite{AP2}).

 But, clearly, there are aspects of this theory which cannot be dealt with in an arbitrary  entropic variety $\BV$. For example, the concepts of primitive element or Lie algebra require that every $\BV$-algebra $A$ is an internal monoid in $\BV$, while the familiar equivalence of the various descriptions of the Sweedler dual of a $k$-algebra depends on the fact that the varieties of $k$-vector spaces are semi-additive.  
 Since these conditions turn out to be equivalent for an entropic variety, it is natural to pay special attention to Hopf monoids in these categories. 

Noting that entropic semi-additive varieties can be viewed  as entropic  J\'onsson-Tarski varieties   or  categories of semimodules over commutative semirings, respectively  (see Proposition \ref{prop:entrJT} below), one observes that this has to some extent been done before in the recent paper \cite{Rom}. %Somewhat surprisingly, 
However, this paper does not deal with the following questions which arise naturally in its context:
%though the concept of primitive element is considered being a central topic. 
\begin{enumerate}
\item Can one generalize the underlying Lie algebra of an algebra and the enveloping algebra of a Lie algebra?
\item Can one generalize the adjunction between bialgebras and Lie algebras, determined by the primitive element functor?
\end{enumerate}
In this note we therefore consider these questions. %In doing so we also correct some inaccuracies occuring in \cite{Rom}.

In spite of the title of this note
we avoid talking about Hopf algebras and bialgebras in varieties (except when these are module categories) and prefer the terms Hopf monoid and bimonoid, respectively,  in order to avoid possible confusion,  since the objects of the monoidal categories under consideration in this note already are called algebras.

The paper is organized as follows:

In Section \ref{sec:basics}, after briefly recalling some fundamentals about entropic varieties, we characterize entropic semi-additive varieties as entropic  J\'onsson-Tarski varieties, or, equivalently, as categories of semimodules over commutative semirings. Moreover, we define the abelian core of an  entropic  J\'onsson-Tarski variety and analyze its monoidal structure. This preliminary section is complemented by a couple of results, which will be used later.

Section \ref{sec:env} starts with the  introduction of tensor bimonoids  and Lie algebras in entropic  J\'onsson-Tarski varieties.  We then define a generalization of the familiar underlying Lie algebra of an algebra and show that the respective functor has a left adjoint, as in the case of modules.

 Section \ref{sec:prim} starts with a discussion of primitive elements in a more conceptual way than this is done in \cite{Rom}; for example, the algebra of primitive elements of a bialgebra in an  entropic  J\'onsson-Tarski variety $\BV$  is characterized as an equalizer in the variety
$\BV$. We then construct a generalization of the familiar primitive element functor and analyze its adjunction properties.

\section{Preliminaries}\label{sec:basics}
\subsection{Terminology}
By a variety $\BV$ we mean a finitary one-sorted variety, considered as a concrete category over $\Set$, the category of sets. Up to concrete equivalence, $\BV$ this is the same as the category of product preserving functors $A\colon\mathcal{T}\lra\Set$, where $\mathcal{T}$ is an algebraic theory	 (see e.g. \cite{AR},\cite{Bor}). 
An $n$-ary term, thus, can be thought of  as a $\mathcal{T}$-morphism $t\colon n\lra 1$ and its interpretation in an algebra $A$ is $t^A:= A(t)$. Recall that   $\mathcal{T}$ is equivalent to the dual of the full subcategory of $\BV$ spanned by all finitely generated free $\BV$-algebras $Fn$; thus, one may think of an $n$-ary term as a $\BV$-homomorphism $F1\lra Fn$.
Every variety  is a locally finitely presentable	category. 

Throughout we make use of the following convention: 
Given an element $x$ of a $\BV$-algebra $A$, the $\BV$-homomorphism $F1\lra A$ with $1\mapsto x$ will be denoted by $x$ as well.

\subsection{Entropic varieties}

It is well known (see \cite{Bor} or \cite{Davey}), that every variety	whose theory	is commutative, 
is a symmetric monoidal closed category; following \cite{Davey} we call any such symmetric monoidal closed category   $\BV$ an \em entropic variety\em. The tensor product of $\BV$, called the \em entropic tensor product\em,  is given by universal bimorphisms in the sense of \cite{BN}. In more details, for algebras	$A$ and $B$ in $\BV$ their tensor product $A \ot B$ is characterized	by the fact, that there is a bimorphism $A\times B\xra{ -\ot -} A\ot B$ over which each bimorphism $A\times B\rightarrow C$ factors uniquely as $f = g\circ (-\ot  -)$ with a homomorphism $g\colon A\ot B\rightarrow C$. 
The internal hom-functor of $\BV$ is given by the $\BV$-algebra $[A,B]$ of all $\BV$-homomorphisms from $A$ to $B$, considered as a subalgebra of $B^A$. In an entropic variety, all operations are homomorphisms. 
The variety $_c\mathbf{Monoids}$ of commutative monoids is a paradigmatic example of an entropic variety.

We will make  use of the following  constructions with respect to an entropic variety $\BV$.
\begin{enumerate}
\item $\Mon\BV$, the category of  $\BV$-monoids $\ma=(A,A\ot A\xra{m}A,F1\xra{e}A)$ in $\BV$. $\Mon\BV$ contains the category $_c\Mon\BV$ of commutative $\BV$-monoids as a full reflective subcategory and the latter is an entropic variety.
\item $\mathsf{Sg}\BV$, the category of  $\BV$-semigroups\footnote{By this we mean the obvious generalization of $\BV$-monoids (see \cite{PP}).} $\ma=(A,A\ot A\xra{m}A)$ in $\BV$. 
and $_c\mathsf{Sg}\BV$, the category of commutative  $\BV$-semigroups.
\item $\Comon\BV$,  the category of  $\BV$-comonoids $\mc=(C,C\xra{ \mu}C\ot C,C\xra{ \epsilon}F1)$.
\item $\Bimon\BV$, the category of  $\BV$-bimonoids $\mb = (B,m,e, \mu, \epsilon)$.
\item $\Hopf\BV$, the category of  $\BV$-Hopf monoids $\mh = (H,m,e, \mu, \epsilon,S)$.
\item $\mathsf{Mod}_\ma$,  the category of right $\ma$-modules $(M,M\ot A\xra{l}M)$ in $\BV$, for any $\BV$-monoid $\ma$; this again is an entropic variety, if the monoid $\ma$ is commutative.\end{enumerate}

\subsection{Entropic  semi-additive varieties}

A  variety 
$\BV$ is  called  \em semi-additive \em or \em linear \em if it is enriched	over the monoidal closed category $_c\mathbf{Monoids}$.  Alternatively, these varieties  can be characterized as being 
 pointed (that is, they have a zero object) and having binary biproducts. These are precisely the J\'onsson-Tarski  varieties whose binary J\'onsson-Tarski operation $+$ satisfies the axiom 
\begin{equation}\label{eqn:linear}
t(x_1,\ldots,x_n) + t(y_1,\ldots,y_n) = t(x_1+y_1,\ldots,x_n+y_n)
\end{equation}
for each $n$-ary operation $t$ (see \cite[1.10.8]{BB}). 
In particular,  
$+$  is a homomorphism. An entropic semi-additive variety $\BV$ has  a unique nullary operation $0$,  every $\BV$-algebra $A$ contains $\{0\}$ as a one-element subalgebra and the constant maps with value $0$ are homomorphisms.

Every such variety is equivalent to a category of $S$-semimodules over some semi\-ring\footnote{{In this note by a \em semiring \em is meant what also (and more appropriately) is called a \em rig\em, meaning, a ring without negatives. We only use the traditional notion since it still seems to be more common.}}  $S$. In fact, 
thinking of an $n$-ary operation symbol as a $\BV$-homomorphism $F1\xra{\omega}Fn$ one concludes  that $\omega = f_1+\cdots + f_n$ with endomorphisms $f_1,\ldots, f_n \in S$, the endomorphism monoid  of the free $\BV$-algebra $F1$, 
since $\BV$ has biproducts (see \cite[1.10.8]{BB}).
$S$ also is a commutative monoid, with addition defined pointwise, by enrichment of $\BV$ over $_c\mathbf{Monoids}$  (that is, by the J\'onsson-Tarski operation $+$). Since every endomorphism preserves $+$, $S$ is a semiring. 
Now $S$ acts on the underlying set $|A|$ of a $\BV$-algebra by $s\cdot a:= s^A(a)$ and this makes $A$ an $S$-semimodule, by Equation (\ref{eqn:linear}). 
Consequently, the following holds,
 since every category $\mathbf{SMod}_S$ of $S$-semimodules over some commutative semi\-ring $S$ is well known to be entropic (see e.g. \cite{Gol}).
\begin{prop}\label{prop:entrJT}
The following are equivalent for a variety $\BV$.
\begin{enumerate}
\item $\BV$ is an entropic   semi-additive variety.
\item $\BV$ is an entropic   J\'onsson-Tarski variety.
\item $\BV$ is equivalent to the variety of $S$-semimodules over some commutative semi\-ring $S$.
\end{enumerate}
\end{prop}
\begin{exas}
The following varieties are  entropic  semi-additive varieties. 
\begin{enumerate}
\item $\mathbf{Ab}$, the category of abelian groups and, more generally, $\Mod_R$, for any commutative unital ring $R$.
\item $\mathbf{SMod}_S=\mathsf{Mod}_S$, where $S$ is a 
commutative monoid in the entropic variety $_c\mathbf{Monoids}$; this is the category of all $S$-semimodules, for a commutative unital semiring $S$. 
\item $\mathbf{SLat_0}$, the variety of   lower bounded 
semilattices. $\mathbf{SLat}$, the variety of   (join)  
semilattices, is entropic but not semi-additive.
 \item $\mathbf{DLat_0} = {_c\mathsf{Sg}\mathbf{SLat_0}}$, the variety of   lower bounded distributive lattices\footnote{The equations used here follow from \cite{Davey}.}.
\end{enumerate}
\end{exas}

\begin{fact}\label{fact:pi}\rm
Let $A$ and $B$ be algebras in an entropic J\'onsson-Tarski variety $\BV$ and $b\in B$. Then there are homomorphisms
\begin{enumerate}
\item  $b_r\colon A\simeq A\ot F1\xra{ A\ot b}A\ot B$ and  $b_l\colon A\simeq F1\ot A\xra{ b\ot A}B\ot A$,\\
for each $a\in A$ one has $b_r(a) =a\ot b $ and $b_r(a) =b\ot a $;
\item $\bar{b}\colon A\xra{\langle b_r,b_l\rangle}(A\ot B)\times (B\ot A)$ with coordinates $b_r$ and $b_l$,\\
for each $a\in A$ one has $\bar{b}(a) =(a\ot b,b\ot a)$;
\item $\pi_A\colon A\xra{\bar{e} }(A\ot A)\times (A\ot A)\xra{ +}A\ot A$, 
if $(A,m,e)$ is a $\BV$-monoid;\\
for each $a\in A$ one has $\pi_A(a) =a\ot 1 + 1\ot a$.\footnote{We here use the notation $e(1) = 1$ in view of the usual notation in the context of primitive elements.}
\end{enumerate}
\end{fact}

\begin{lem}\label{lem:em}
Let  $A$ be an algebra in an  entropic    J\'onsson-Tarski variety $\BV$. Then the following hold.
\begin{enumerate}
\item The triple $(A,A\times A\xra{+^A}A, 0)$ is a commutative  internal monoid in $\BV$ and, thus, a (commutative) monoid. This is the only internal monoid on $A$ by the Eckmann-Hilton argument.
By this construction $\BV$ is isomorphic to the category  of commutative internal monoids in $\BV$ (see \cite[1.10.5]{BB}).
\item An $n$-fold sum $x+\cdots +x$ in  $A$ can be written as $n\cdot x$ with $n$ the $n$-fold sum $1+\cdots +1 \in F1$; that is, $x+\cdots +x$ is the value of $1+\cdots +1$ under the canonical isomorphism $F 1\ot  A \simeq A$. 
\item The variety $\BV$ is isomorphic to the category  $\mathsf{Mod}_\mathsf{F1}$ of modules of the commutative monoid $\mathsf{F1}$ $($see \cite{Davey}$)$. 
\end{enumerate}
\end{lem}

An element $x$ of an algebra $A$ in a  J\'onsson-Tarski   variety is called \em invertible\em, if it is invertible in the monoid $(A,A\times A\xra{+^A}A, 0)$, that is, if there is an element $y\in A$ with $x+y = 0\ (=y+x)$; such an element is  uniquely determined and will be denoted by $-x$.   $Inv(A)$ denotes the set of  invertible elements of $(A,+^A, 0)$.

\begin{lem}\label{lem:inv}  
Let $\BV$ be an entropic   J\'onsson-Tarski   variety.
Then, for every $\BV$-algebra $A$, the following holds. 
\begin{enumerate}
\item $Inv(A)$ is a $\BV$-subalgebra  of $A$ with an embedding $\upsilon_A$ and, hence, an internal submonoid  of $(A,+^A, 0)$.
\item Every $\BV$-homomorphism $A\xra{f}B$ is, by restriction and corestriction, a $\BV$-homomorphism $Inv(A)\xra{}Inv(B)$. 

In particular, for every $b\in B$, the homomorphism $b_r\colon A\xra{}A\ot B$ restricts to a homomorphism $Inv(A)\xra{}Inv(A\ot B)$ and, hence, the homomorphism  $\upsilon_A\ot \upsilon_B$ factors 
as $Inv(A)\ot Inv(B)\xra{\upsilon_{A,B}}Inv(A\ot B)\hookrightarrow A\ot B$; moreover one has  $m(a\ot b) \in Inv(A)$, 
 for every monoid $(A,m,e)$ in $\BV$ and for all $a,b\in Inv(A)$. 
\item  For every $\BV$-homomorphism $A\xra{f}B$ one has $f(-x) = -f(x)$, for each $x\in Inv(A)$.
\item The map $Inv(A)\xra{i}Inv(A)$ with $x\mapsto -x$ is a $\BV$-homomorphism. Consequently, 
$Inv(A)$ is a (commutative) internal group and, in fact the largest such contained in $A$. 
\end{enumerate}
\end{lem}

\begin{proof}
For every $k$-ary term $t$ and  invertible elements $m_1,\ldots,m_k\in A$   the element \linebreak $t^{A}(m_1,\ldots,m_k)\in A$ is invertible, since {\small $$t^{A}(m_1,\ldots,m_k)+ t^{A}(-m_1,\ldots,-m_k) =  t^{A}(m_1+ (-m_1),\ldots,m_k+ (-m_k)) = t^A(0,\ldots, 0)=0.$$}
This calculation shows, moreover, that the map $Inv(A)\xra{i}Inv(A)$ with $x\mapsto -x$ is a $\BV$-homomorphism.
Thus, $Inv(A)$ is an internal subgroup in $\BV$ and, when considered as an internal monoid,  it is an internal submonoid of $A$. 

Obviously $f(-x) = -f(x)$ for every homomorphism $A\xra{f}B$ and every $x\in Inv(A)$, which proves items 2 and 4.

The rest is trivial. 
\end{proof}

We denote by  $\BV_\Ab$  the full subcategory  of  $\BV$ spanned by all $\BV$-algebras $A$ with $Inv(A)=A$. 
 If $\BV = \Alg(\Omega,\mathcal{E})$, then 
$\BV_\Ab$ is the variety $\Alg(\Omega',\mathcal{E'})$ with $\Omega'$  obtained from $\Omega$ by adding a unary operation $-$,  and $\mathcal{E'}$  obtained from $\mathcal{E}$ by adding the equations 
\begin{enumerate}
\item $x +(-x) = 0$,
\item $\omega(-x_1,\ldots, -x_n) = -\omega(x_1,\ldots, x_n)$ for all $n$-ary operations  $\omega\in\Omega$, for all $n\in\N$.
\end{enumerate}
Obviously,  $\BV_\Ab$  
is an entropic J\'onsson-Tarski variety and 
coincides with 
the category of all internal groups in $\BV$.  
Consequently, $\BV_\Ab$ is an additive category (see \cite[1.10.13]{BB}) and, in fact,  
 the largest additive subvariety of $\BV$. Being exact as a variety, $\BV_\Ab$ even is an abelian category (see \cite[2.6.11]{Bor}).
In accordance with \cite{vdL} we call  $\BV_\Ab$ the \em abelian core \em of $\BV$.

\begin{prop}\label{prop:AbV}
For every entropic  J\'onsson-Tarski variety  $\BV$ the following hold.
\begin{enumerate}
  \item  $\BV_\Ab$ is a full isomorphism-closed reflective subcategory of $\BV$.
\item The assignment $A\mapsto Inv(A)$ defines a functor  $\BV\xra{Inv}\BV_\Ab$ and this is right adjoint to the embedding $\BV_\Ab\hookrightarrow\BV$. Moreover,   $\BV_\Ab\hookrightarrow\BV\xra{Inv}\BV_\Ab = Id$.
\item $\BV_\Ab$ is closed under the entropic tensor product $-\ot_\BV -$ of $\BV$; consequently, the entropic monoidal structure  of $\BV_\Ab$ is given by $-\ot_\BV -$ and the reflection $RF1$ of $F1$ into $\BV_\Ab$ as the unit object. 
\item The embedding $\BV_\Ab\hookrightarrow\BV$ is a symmetric monoidal functor.
\end{enumerate}
\end{prop}

\begin{proof}
 $\BV_\Ab$ is a full isomorphism-closed  subcategory of $\BV$ by the preceding lemma. Since its embedding into $\BV$ commutes with the forgetful functors, it is an algebraic functor and, thus, has a left adjoint.
 Obviously every morphism of internal groups $f\colon G\lra H$ factors over the embedding $Inv(A)\hookrightarrow A$, which shows  that $Inv$ is a coreflection. This proves items (1) and~(2).

 Denoting the entropic tensor product of $\BV_\Ab$  by $-\ot -$ and that of $\BV$ by $-\ot_\BV -$, we first deduce from Lemma \ref{lem:inv} that, for
 internal groups $G$ and $H$, all elements $g\ot_\BV h\in G\ot_\BV H$ are invertible  since the map $g\ot_\BV -$ is a homomorphism. Thus, $G\ot_\BV H$ is an internal group.
 We then have, for every triple $G,H,K$ of internal groups in $\BV$, 
$ \BV_\Ab(G\ot H, K)\simeq \BV_\Ab(G,[H,K]) =  \BV(G,[H,K])\simeq \BV(G\ot_\BV H, K)\simeq \BV_\Ab(G\ot_\BV H, K)$, since the internal hom-functors of $\BV$ and $\BV_\Ab$ coincide. 

To complete the proof of items (3) and (4) it remains to show that the following diagram commutes for every internal group $G$, where $F1\xra{r}RF1$ is the reflection map. 
\begin{equation}\label{diag:ref}
\begin{aligned}
%\ar@{.>}[dl]^d
%\ar@/^2pc/@{->}[rr]^m
\xymatrix@=2.5em{
 F1\ot_\BV G    \ar[r]^{\!\!\!\!\! r\ot_\BV id  }\ar[d]_{can_\BV  }&{RF1\ot_\BV G } \ar@{=}[d]^{ }  \\
  G            &   RF1\ot G \ar[l]^{\!\!\!\!\! \!\!\! can} 
}
\end{aligned}
\end{equation}
But this is clear, since $ F1\ot_\BV G$ is generated by the elements $1\ot g$, $g\in G$,  and $r$ maps the free generator of $F1$ to the free generator of the free $\BV_\Ab$-algebra $RF1$.
\end{proof}

\begin{exas}\hfill
\begin{enumerate}
\item $(_c\mathbf{Monoids})_\Ab= \mathbf{Ab}$.
\item $_c\Mon\BV_\Ab$ is isomorphic to $(_c\Mon\BV)_\Ab$. \\
In fact, $(_c\Mon\BV)_\Ab$ is 
the full subcategory of $_c\Mon\BV$, consisting of all commutative monoids $(M,m,e)$ with $Inv(M) = M$ and, since the embedding $\BV_\Ab\xrightarrow{E}\BV$ is monoidal, it embeds $\Mon\BV_\Ab$ into this category. Conversely, if $(M,m,e)\in{_c\Mon\BV}$ satisfies $Inv(M) = M$ and $F1\xra{r}RF1$ is the reflection of $F1$, denote by  
$RF1\xra{e'}M$  the unique homomorphism with $e'\circ r = e$;  then $(M,m, e')\in \Mon\BV_\Ab$ (use Lemma \ref{lem:inv} and Diagram (\ref{diag:ref})) and $(M,m,e) = E(M,m, e')$. 
\item $(\Mod_R)_\Ab = \Mod_R$, for every  commutative ring $R$ and, more generally,
\item $(\mathbf{S}\Mod_S)_\Ab= \Mod_{RS}$, for every  commutative semiring $S$, where $RS$ is the reflection of $S$ into the category of commutative rings\footnote{The embedding of the category of commutative rings into the category of commutative semirings is an algebraic functor and therefore has a left adjoint $R$.}. \\
This is easily seen when recalling the fact that, in every entropic variety $\BV$, the left $\ma$-modules $(M,A\ot M\xra{ l}M)$ of a $\BV$-monoid $\ma$ are in one-to-one correspondence with  $\BV$-monoid morphisms $\ma\xra{\phi}[M,M]$, where $\phi$ corresponds to $l$ by the adjunction $-\ot M \dashv[M,-]$ (see e.g. \cite{Por}). Thus, an $S$-semimodule $M$ with $M = Inv(M)$ is a semiring homomorphism $S\xra{\phi}[M,M]$, where $[M,M]$ is the endomorphism monoid of $M$ in $_c\mathbf{Monoids}$. 
 This is a monoid in $\mathbf{Ab}$ by item (1) 
and, thus, $\phi$ corresponds to a unique ring homomorphism $RS\xra{\phi}[M,M]$, that is, to an $RS$-module.
\item $(\mathbf{SLat_0})_\Ab =\{ 0\}$.
\end{enumerate}
\end{exas}

\section{The universal envelope functor}\label{sec:env}
\subsection{Tensor bimonoids and Lie algebras}\label{ssec:tensoralg}

Let $\BV$ be an entropic variety.
By the standard construction of free monoids in monoidal closed categories (see [18])  the free monoid $TA$ in $\Mon\BV$ over a $\BV$-algebra
$A$ has $TA = \coprod_{n\in\N}A^{\ot n}$ as its underlying $\BV$-algebra\footnote{$A^{\ot n}$ is to be understood in the obvious way: $A^{\ot 0} = F1$ and $A^{\ot n+1} = A \ot A^{\ot n}$} 
with unit $F1\xra{ \iota_0}TA$ and multiplication given by ``concatination".
The coproduct injection $\iota_1 \colon A \lra TA$ is its universal morphism.

 If  $\BV$ is an entropic  J\'onsson-Tarski variety, this $\BV$-monoid becomes a $\BV$-bimonoid $(TA,\mu,\epsilon)$, called the \em $\BV$-tensor bimonoid\em,  as follows, where $\pi_{TA}$ is  the homomorphism defined  in Fact \ref{fact:pi}. 
\begin{enumerate}
\item $\mu\colon TA \lra  TA \ot TA$ is the homomorphic extension of the $\BV$-homomorphism
$A \xra{\iota_1} TA \xra{\pi_{TA}} TA \ot TA$  to a morphism of $\BV$-monoids. 

\item $\epsilon\colon TA \lra F1$ is the homomorphic extension the $\BV$-homomorphism $A \xra{0} F1$ to a morphism of $\BV$-monoids. 
\end{enumerate}
That $(TA,\mu,\epsilon)$ this way becomes a $\BV$-comonoid can be shown literally the same way is in the case of modules. Since $\mu$ and $\epsilon$ are $\BV$-monoid morphisms by definition, $TA$ is a $\BV$-bimonoid.

For every  $A\in \BV_\Ab$
the tensor bimonoid $(TA,\mu,\epsilon)$ even becomes a Hopf monoid in $\BV_\Ab$ and, thus, in $\BV$.  The required antipode acts as $a_1\ot\cdots a_n\mapsto (-1)^na_n\ot \cdots\ot  a_1$\footnote{Note that the notation $(-1)^nx$ is symbolic and short for $(-(-\cdots (-x)\cdots)$.}.   The proof again is literally the same as in the case of modules.%$

Every Hopf algebra $\mh$, has an underlying Lie algebra, obtained as the underlying Lie algebra of  the underlying algebra of $\mh$. This construction is not possible over an arbitrary entropic variety $\BV$, since neither can the Jacobi identity
\begin{equation}\label{eqn:jac}
[x,[y,z]]+[y,[z,x]]+[z,[x,y]]=0
\end{equation}
be expressed, nor can the underlying Lie algebra of a monoid  be defined.

\begin{defi}\label{def:Lie}\rm
Let  $\BV$ be an entropic  J\'onsson-Tarski variety. A \em $\BV$-Lie algebra \em is a pair $(A, [-,-])$ consisting of a $\BV$-algebra $A$ and a $\BV$-bimorphism $[-,-]\colon A\times A\lra A$ satisfying the identity (\ref{eqn:jac}) and, in addition, the identity $[x,x] = 0$. \\
A \em $\BV$-Lie morphism \em $(A, [-,-])\xra{f}(B, [-,-])$ is a $\BV$-homomorphism $A\xra{f}B$ making the following diagram commute.
\begin{equation}\label{diag:Lhom}
\begin{aligned}
%\ar@{.>}[dl]^d
%\ar@/^2pc/@{->}[rr]^m
\xymatrix@=2em{
    A\times A \ar[r]^{f\times f  }\ar[d]_{[-,-]  }&{B\times B } \ar[d]^{[-,-]  }  \\
  A   \ar[r]_{ f}          &    B 
}
\end{aligned}
\end{equation}
This defines the category $\Lie\BV$  of  $\BV$-Lie algebras.   There is a forgetful functor $\Lie\BV\xra{V}\BV$.\end{defi}
Denoting by $[-]\colon A\otimes A\lra A$ the $\BV$-homomorphism corresponding to the bimorphism $[-,-]$  
 one may say, equivalently,  that  $(A, [-])$ is a $\BV$-Lie algebra, if $[-]$  satisfies the  axioms
\begin{equation*}\label{eqn:jac2}
[x,[y\ot z]]+[y,[z\ot x]]+[z,[x\ot y]]=0  \text{\ and \ } [x\ot x] = 0.
\end{equation*}
In this language the Lie homomorphism axiom obviously is commutativity of Diagram (\ref{diag:Lhom}) with $\times$ replaced by $\ot$ and the bimorphism $[-,-]$ replaced by $[-]$.

Obviously, $\Lie\BV$  is a  J\'onsson-Tarski variety and  $(\Lie\BV)_\Ab= \Lie\BV_\Ab$.

\subsection{The enveloping (bi)monoid of a Lie algebra}

Since there is no underlying functor from $\BV$-monoids to $\BV$-Lie algebras for  an arbitrary entropic 
J\'onsson-Tarski variety, unless the theory of $\BV$ admits an inverse of the J\'onsson-Tarski operation $+$, that is, if $\BV\simeq\Mod_R$,    the standard categorical argument for the existence of the universal enveloping  algebra $UL$ of a Lie-algebra $L$  does not apply\footnote{An underlying functor would be algebraic  and, thus, have a left adjoint.}.
 We show next that
one  can in spite of that generalize  the standard construction, where the resulting $\BV$-monoid even carries the structure of a $\BV$-bimonoid, as in the case of $\BV = \Mod_R$.

\begin{thm}\label{prop:U}
There exist functors $Lie\colon {\Mon\BV} \lra \Lie\BV$ and $U\colon \Lie\BV\xra{}\Mon\BV$, where we call $Lie\ma$ the \em underlying Lie algebra of the monoid $\ma$\em, and $U\ml$ the \em enveloping monoid of the Lie algebra $\ml$\em, such that 
\begin{enumerate}
\item the  diagram commutes, that is, $Lie$ takes its values in $\Lie\BV_\Ab$
\begin{equation*}%\label{diag:}
\begin{aligned}
%\ar@{.>}[dl]^d
%\ar@/^2pc/@{->}[rr]^m
\xymatrix@=2em{
   {\Mon\BV} \ar[d]^{  }\ar[r]^{Lie }& { \Lie\BV} \ar[d]^{ }  \\
      \BV\ar[r]^{Inv}    &  \BV
}
\end{aligned}
\end{equation*}
%commutes and, hence, $Lie$ takes its values in $\Lie\BV_\Ab$.
\item there is a natural quotient $q\colon T\circ V\Rightarrow U$
\begin{equation*}
\begin{aligned}
%\ar@{.>}[dl]^d
%\ar@/^2pc/@{->}[rr]^m
\xymatrix@=2em{
  \Lie\BV \ar[rr]^U \ar[dr]_{V}&{}&{\Mon\BV }   \\
& \BV\ar[ur]_{T  }\ar@{=>}[u]^q %\ar@{<=}"1";"2"^-{q}Lie
           &    }
\end{aligned}
\end{equation*}
\item The restriction ${U_\Ab}\, \colon\Lie\BV_\Ab\hookrightarrow \Lie\BV\xra{U}\Mon\BV$ of $U$ is left adjoint to the corestriction 
$Lie_\Ab\colon  \Mon\BV\xra{} \Lie\BV_\Ab$ of $Lie$. Consequently, $Lie$ is right adjoint to $U_\Ab\circ R$, where $R$ is the reflection of $\Lie\BV$ into $\Ab(\Lie\BV) = \Lie\BV_\Ab$.
\item $U$ factors as $\Lie\BV\xra{U_{Bi}}\Bimon\BV\xra{|-|}\Mon\BV$, that is, the enveloping monoid of a Lie algebra $\ml$ carries a bimonoid structure and this construction is functorial.
\item $U_\Ab$ factors as  $\Lie\BV_\Ab\xra{U_H}\Hopf\BV\xra{}\Mon\BV$,  the enveloping monoid of a Lie algebra $\ml$ carries a Hopf monoid structure in a functorial way, provided that $Inv L = L$.
\end{enumerate}
\end{thm}

\begin{proof}
Consider, for a $\BV$-monoid  $\ma = (A,m,e)$, the map $[-]\colon Inv(A)\ot Inv(A) \lra Inv(A)$ given by $[a\ot b]: = m(a\ot b) - m(b\ot a)$. This is a $\BV$-homomorphism by item (2) of Lemma \ref{lem:inv}. $Lie(A,m):= (Inv(A),[-])$ then  is a Lie algebra in $\BV$ (in fact in $\BV_\Ab$) and, for every $\Mon\BV$-morphism $f\colon (A,m,e)\lra (B,n,u)$ its restriction $Lie(f)$ to a homomorphism  $Inv(A)\lra Inv(B)$ satisfies $f([a\ot b] = [(f\ot f)(a\ot b)]$, for all $a,b\in A$. This defines the functor $Lie$.

Given a $\BV$-Lie algebra $(L,[-])$, form the free $\BV$-monoid $(TL,m,e)$ over $L$ as in Section \ref{ssec:tensoralg}. 
Let $\rho\colon L\ot L\lra TL\times TL$ be the $\BV$-homomorphism with $\pi_1\circ \rho = m\circ\sigma\circ (\iota_1\ot \iota_1)$ and $\pi_2\circ \rho =\iota_1\circ [-]$. 
Then the family $\mathcal{S}_L$ of all $\BV$-monoid morphisms $TL\xra{f_i}\ma_i$  with 
$f_i\circ  m \circ (\iota_1\ot \iota_1)  =  f_i \circ (+^L\circ \rho)$  
has a  (regular epi, monosource)-factorization 
$ TL\xra{f_i} A_i =  TL\xra{q_L} U\ml \xra{m_i} A_i$ 
in $\Mon\BV$, since $\Mon\BV$ is a variety. If $l\colon \ml\lra \ml'$ is a Lie-morphism, the $\BV$-homomorphism $Tl\colon TL\lra TL'$ is a monoid homomorphism and, for each $f'_i \in \mathcal{S}_{L'}$, $f'_i\circ Tl\in \mathcal{S}_L$. Thus, there exists a unique monoid morphism $Ul\colon UL\lra UL'$ with $Ul\circ q_L = q_{L'}\circ TL$. This defines a functor $U$ as well as a natural transformation $q\colon T\circ |-|\Rightarrow U$ being pointwise a quotient.
We call the monoid $U\ml$ the \em enveloping monoid of $\ml$\em.

Denote, for $\ml\in\Lie(\BV_\Ab)$,  by $\eta\colon L\xra{} Lie UL$ the $\BV$-homomorphism  $q_L\circ\iota_1$ ($L\in \BV_\Ab$ implies $q_L\circ\iota_1(c)\in Inv(UL)$, for each $x\in L$). 
Since $q_L(m(\iota_1 x\ot \iota_1y))=q_L(m(\iota_1 y\ot \iota_1 x)+\iota_1[x\ot y])=q_Lm(\iota_1 y\ot \iota_1 x) +q_L\iota_1[x\ot y]$ by definition of $q_L$,  $[q_L\iota_1 x\ot q_L\iota_1 y] = m(q_L\iota_1 x \ot q_L\iota_1 y) - m(q_L\iota_1 x \ot q_L\iota_1 x) $ by definition of $Lie$, and since $q_L$ is a monoid morphism, one concludes that
$\eta$ is a $\Lie\BV$-morphism.

For any Lie-morphism $\ml\xra{f}Lie\ma$ with 
$\ma =(A,m,e)\in \Mon\BV$ let $TL\xra{f^\sharp}A$ be its extension to a monoid morphism. By the definition of $U$ this morphism factors as $f^\sharp = TL\xra{q_L}UL\xra{\tilde{f}} \ma$ with a monoid morphism $UL\xra{\tilde{f}}\ma$, which is the unique such morphism with $Lie\tilde{f}\circ \eta = f$. This proves the first statement of item 3. The second statement follows by composing this adjunction with the adjunction given by the embedding $(\Lie\BV)_\Ab\hookrightarrow \Lie\BV$ and its left adjoint (see Proposition \ref{prop:AbV}).

The monoid $U\ml$ can be supplied with a bimonoid structure as follows.
Let $u\colon L\lra UL\ot UL$ be 
map with $x\mapsto (q\circ\iota_1) x\ot 1 + 1\ot (q\circ\iota_1) x$. 
 In other words, with notation as in Section \ref{ssec:tensoralg}, $u$ is the $\BV$-homomorphism 
 $$L\xra{\iota_1}TL\xra{\mu}TL\ot TL\xra{q\ot q} UL\ot UL = L\xra{\pi}L\ot L\xra{\iota_1\ot\iota_1}TL\ot TL\xra{q\ot q}UL\ot UL,$$ 
which has  $\nu :=  TL\xra{\mu} TL\ot TL\xra{q\ot q}UL\ot UL$ as its unique extension to a $\BV$-monoid morphism. 
Now the  straightforward calculation
{\footnotesize
\begin{eqnarray*}
u([x\ot y] + m(y\ot x)) %=&  %(q\ot q)\circ(\iota_1\ot\iota_1) ([x\ot y] + m(y\ot x)) = \pi([x\ot y] + m(y\ot x)\\
&=& (q\ot q)\circ (\iota_1\ot\iota_1)\big(([x\ot y] + m(y\ot x))\ot 1 + 1\ot ([x\ot y] + m(y\ot x))\big)\\
&=& \big(q\circ\iota_1([x\ot y] + m(y\ot x))\big)\ot 1 +1\ot \big(q\circ\iota_1([x\ot y] + m(y\ot x))\big)\\
&=&\big( q\circ\iota_1(m(x\ot y))\big)\ot 1 + 1\ot \big(q\circ\iota_1(m(x\ot y))\big)\\
&=&(q\ot q)\circ (\iota_1\ot\iota_1)\big(m(x\ot y)\big)\\
&=&u(m(x\ot y))
\end{eqnarray*}
}
with $x \ot y\in L\ot L$ proves  the  equation 
\begin{equation*}\label{eqn:fi}
u\circ m = \nu\circ \iota_1\circ m = \nu\circ \iota_1\circ (+\circ\rho) = u\circ (+\circ\rho)
\end{equation*}
which shows that $\nu$ belongs to $\mathcal{S}_L$.
 Consequently there exists a morphism of $\BV$-monoids $UL\xra{\delta}UL\ot UL$, 
such that the following diagram commutes
\begin{equation}\label{diag:q}
\begin{aligned}
%\ar@{.>}[dl]^d
%\ar@/^2pc/@{->}[rr]^m
\xymatrix@=2em{
   TL  \ar[r]^{ q }\ar[d]_{\mu  }&{UL } \ar[d]^\delta   \\
 TL\ot TL   \ar[r]_{ q\ot q}          &    UL\ot UL 
}
\end{aligned}
\end{equation}
Since the extension  $\epsilon\colon TL\lra F1$ of the $0$-homomorphism $L\lra F1$ belongs to the family $(f_i)_i$ as well, as is easily seen, 
$\epsilon$ factors in $\Mon\BV$ as $TL\xra{q}UL\xra{\upsilon}F1$.
It now follows trivially that $(TL,\mu,\epsilon)\xra{q}(UL,\delta,\upsilon)$ is a morphism of $\BV$-bimonoids, since $q$ is surjective.
This construction is functorial:
If $f\colon L\xra{}L'$ is a Lie-morphism and $Tf$ the corresponding monoid-morphism $TL\xra{}TL'$, then $q_{L'}\circ Tf\in \mathcal{S}_L$, such that  there is a unique monoid-morphism $Uf\colon UL\xra{}UL'$ with $Uf\circ q_L = q_{L'}\circ Tf$. $Uf$ is a comonoid morphism as well; it is compatible with the comonoid structures just defined, as is easily seen. This proves item 4.

Item 5 now follows literally as in the case of modules: the required antipode $S$ is the extension of the $\BV$-homomorphism $L\xra{}L$ given by $x\mapsto -x$ and this is preserved by the bimonoid morphisms $Uf$ just defined. 
\end{proof}

\section{Primitive element functors}\label{sec:prim}

Recall that in the classical case, where $\BV=\Mod_R$, there exists a so-called the \em primitive element functor \em 
$P\colon\Bialg_R\xra{}\mathbf{Lie}_R$, which is right adjoint to 
$U_{Bi}\colon\mathbf{Lie}_R\xra{}\Bialg_R$. This cannot be generalized to arbitrary entropic varieties, since here one  cannot even define so-called \em primitive elements\em.
This problem  is partly addressed in \cite{Rom} for  J\'onsson-Tarski varieties.
The problem whether the functor in this theorem gives rise to an adjunction is not considered.
We here deal with this problem as follows.

\subsection{Primitive elements}

\begin{defi}%\label{def:}1
\rm
Let $\BV$ be an entropic J\'onsson-Tarski variety. 
An element $p$ of a $\BV$-bimonoid $\mb$ with comultiplication $\mu$ is called \em primitive\em, provided that 
$\mu(p) = \pi_B(a)=p\ot 1 + 1\ot p$ and  $\epsilon(p) = 0$.\footnote{Here $1$ is short for $e(1)$, the image of $1$ under the unit $F1\xra{ e} \mb$.} 
\end{defi}

Note that, by definition of the comultiplication of the tensor bimonoid $TA$, the elements of $A$ (more precisely, the elements $\iota_1(a)$ for $a\in A$) are primitive elements in $TA$.

We next provide  a conceptual description of  primitive elements in bimonoids. 
\begin{prop}\label{corr:prim}
Let $\mb = (B,m,e,\mu,\epsilon)$ be a bimonoid  in an entropic J\'onsson-Tarski variety $\BV$.

Let $E_1$ be the equalizer of the homomorphisms $B\xra{\pi_B}B\ot B$  and $B\xra{\mu}B\ot B$ 
and  $E_2$  the equalizer of the homomorphisms $B\xra{\epsilon}F1$  and $B\xra{0}F1$. % by $E_2$.

  Then the underlying set of the $\BV$-algebra $E_1\cap E_2$ is the set of all  primitive elements. 
In particular, the primitive elements of $\mb$ form a $\BV$-subalgebra $Prim(\mb)$ of $B$. 

 The assignment $\mb\mapsto Prim(\mb)$ defines a faithful functor $Prim\colon \Bimon\BV\lra \BV$.
\end{prop}

\begin{proof}
Only the last statement requires an argument.
If $f\colon\mb\xra{}\mb'$ is a morphism in $\Bimon\BV$ and $p$ is a primitive element in $\mb$, then $\mu'(fp) = (f\ot f)\circ\mu (p) = (f\ot f)(p\ot 1+1\ot p) = fp\ot 1 = 1\ot fp$ by the morphism properties of $f$; hence $f$ yields by restriction and corestriction a $\BV$-homomorphism $Prim(f)\colon Prim(\mb)\lra Prim(\mb')$. This proves functoriality of the construction $Prim$.
\end{proof}

\begin{lem}\label{lem:prim}
Let $(\mh,S)$ be a Hopf monoid in an entropic J\'onsson-Tarski variety $\BV$.  Then 
\begin{enumerate}
\item $S$ can be restricted to a $\BV$-homomorphisms $Inv(H)\xra{}Inv(H)$ and $Prim(\mh)\xra{}Prim(\mh)$.
\item For each $x\in Inv(H)$
one has $Sx = -x$.  
\item $Inv(H)$ contains $Prim(\mh)$ as a $\BV$-subalgebra.
\item $Prim(\mh)$ is an internal group in $\BV$.
\end{enumerate}
 The assignment $\mh\mapsto Prim(\mh)$ defines a faithful functor $Prim\colon \Hopf\BV\lra \BV_\Ab$.
\end{lem}

\begin{proof}
$S$, being a homomorphism, preserves inverses by Lemma \ref{lem:inv}.  
Since
the antipode $S$ of a Hopf monoid $(\mh,S)$ is a bimonoid morphism $\mh\lra\mh^{\op,\Cop}$ and the antipode equation is satisfied, $Sp$ is primitive for each primitive element $p$, by the simple calculation $\mu(Sp) = S\circ\sigma (\mu p) = S(1\ot p+ p\ot 1) =  1\ot Sp + Sp\ot 1$ (recall that $+$ is commutative).

It remains to show that  $Sp$ is an additive inverse of $p$, for each primitive element $p$. In fact, $Sp+p= m(Sp\ot 1)+m(1\ot p) = m(Sp\ot 1 + 1\ot p) = m\circ (S\ot id)(p\ot 1+ 1\ot p)= m\circ (S\ot id)\circ\mu(p) = e\circ\epsilon (p) =0$. 
\end{proof}

\subsection{An adjunction between $\Hopf\BV$ and $\Lie\BV$}

The following result generalizes  part of the famous Milner-Moore Theorem.

\begin{thm}
Let  $\BV$ be an entropic  J\'onsson-Tarski variety. Then there exists a faithful functor $\bar{P}\colon\Hopf\BV\lra \Lie\BV$,   such that the following diagram commutes.
\begin{equation*}
\begin{aligned}
\xymatrix@=2em{
{ \Hopf\BV}\ar[rr]^{\bar{P}} \ar[d]_{Prim}&& \Lie\BV \ar[d]^{|-| }     \\
 {\BV_\Ab }\ar@{^(->}[rr]    & &\BV 
}
\end{aligned}
\end{equation*}
$\Hopf\BV \xra{P_H}\Lie\BV_\Ab$, the corestriction of $\bar{P}$,  is right adjoint to $\Lie\BV_\Ab\xra{U_{H}}\Hopf\BV $. Consequently, $\bar{P}$ %\Hopf\BV \xra{{P}}\Lie\BV$
 is right adjoint to $
\Lie\BV\xra{R}\Lie\BV_\Ab\xra{U_{H}}\Hopf\BV $, where $R$ is the reflection functor.
\end{thm}

\begin{proof}
Given a $\BV$-Hopf monoid $\mh$, one can define $\BV$-homomorphisms  $[-,-]_\mh$ and  ${_\mh[-,-]}$  $Prim(\mh)\otimes Prim (\mh )\lra Prim(\mh)$ 
by 
\begin{equation*}\label{eqn:bracket}
[x,y]_\mh = m(x\ot y) +  m(y\ot Sx) \text{\ \ and \ }  _\mh[x,y] = m(x\ot y) +  m(Sy\ot x )
\end{equation*}
%Define, for each   
 In fact,  
since homomorphisms preserve invertible elements and inverses and since $S$ and  $m(x\ot -)$    are homomorphisms,  by item (2) of Lemma \ref{lem:inv}
the elements $ m(x\ot y)$,  $m(y\ot Sx)$ and  $m(Sy\ot x )$ are invertible,  provided that $x$ and $y$ are primitive, hence invertible (see item (3) of Lemma \ref{lem:prim}). Consequently one has for primitive elements $x$ and $y$
\begin{eqnarray*}
[x,y]_\mh &= &m(x\ot y) +  m(y\ot Sx) = m(x\ot y) +  m(y\ot (-x)) \\
& =& m(x\ot y) +  (-m(y\ot x)) = m(x\ot y) +  Sm(y\ot x) \\
&=& {_\mh[x,y]}
\end{eqnarray*}

Now, using the equation $[x, y]_\mh = m(x \ot y) + (-m(y \ot x))$ just shown to hold, one proves literally the same way as in the case of R-modules
\begin{enumerate}
\item If $x,y \in H$ are primitive, so is $[x,y]_\mh$.
\item $[-,-]_\mh$ 
satisfies Equation (\ref{eqn:jac}) as well as
$[x, x]_\mh = 0$.
\end{enumerate}
This shows that $\bar{P}\mh := (Prim(\mh),[-]_\mh)$ is a Lie algebra in $\BV_\Ab$. Obviously this construction is functorial.

Recall from the construction of the tensor bimonoid that all elements of the form $\iota_1(x)$ are primitive in $TL$. Since $TH\xra{q}UH$ is a bimonoid morphism by commutativity of Diagram (\ref{diag:q}), it follows from Proposition \ref{corr:prim} 
that, for each $L\in \Lie\BV_\Ab$, the $\BV$-morphism $\eta = L\xra{\iota_1}TL\xra{q} Lie UL = Inv UL$ 
factors as $$L\xra{\iota_1'} Prim TL \xra{q'} Prim UL=\bar{P} U_{H}(L)\hookrightarrow Inv(UL) =Lie UL$$
and that $\eta':=q'\circ\iota_1'$ is a Lie-morphism. 

If now, for some Hopf monoid $\mh$,  $\mathsf{L}\xra{f} \bar{P}H$  is a Lie-morphism, so is $f' =L\xra{f}\bar{P}H\hookrightarrow Lie_\Ab H$, such that by item 3 of Proposition \ref{prop:U} 
there exists a unique monoid morphism $\tilde{f}\colon U_\Ab L\lra H$ with $f' = Lie_\Ab(\tilde{f})\circ q\circ\iota_1$. 

From $\im f'\subset Prim(\mh)$ one concludes $\im \tilde{f}\subset Prim(\mh)$, and this implies that the outer frame of the following diagram commutes.

\begin{equation*}
\begin{aligned}
%\ar@{.>}[dl]^d
%\ar@/^2pc/@{->}[rr]^m
\xymatrix@=2.5em{
 L    \ar[r]^{ \iota_1 }\ar[d]_\pi   &   TL \ar[d]^{ \mu}\ar[r]^{q_L}  &   UL\ar[r]^{\tilde{f}} \ar[d]^\delta  & H\ar[d]^{\mu_H} 
 \\
 L\ot L  \ar[r]_{ \iota_1\ot\iota_1}  &  TL\ot TL \ar[r]_{q_L\ot q_L}  &  UL\ot UL \ar[r]_{\tilde{f}\ot \tilde{f}}   &H\ot H
}
\end{aligned}
\end{equation*}
Since the left and middle cells commute by definition of $\mu$ and commutativity of Diagram (\ref{diag:q}), respectively, $\tilde{f}$ is compatible with the comultiplications. Similarly, $\tilde{f}$ preserves counits and so is a morphism in $\Hopf\BV$. This shows that $U_H$ is left adjoint to $\bar{P}$.
\end{proof}

\begin{rems}
It follows from the theorem above that, as in the classical case, we also obtain a functor $\Bimon\BV\lra\Lie\BV$ by forming the composition $I\circ  P_H\circ C$, where $C\colon\Bimon\BV\xra{}\Hopf\BV$ is the coreflection functor, available in every locally presentable monoidal closed category (see \cite{HEP_QMI}), hence in every entropic variety, and $I$ is the embedding $\Lie\BV_\Ab\hookrightarrow \Lie\BV$.
This  can be considered a substitute for $P$, since 
 $I\circ  P_H\circ C\vdash E\circ U_H\circ R = U_{Bi}\circ I \circ R$ by composition of adjunctions and since,  by construction, $U_{Bi}\circ I = E\circ U_H$. Hence, for $\BV=\Mod_R$, where $I= id$, one has $I\circ P_H\circ C = P_H\circ C \vdash U_{Bi}$ and, thus, $I\circ P_H\circ C\simeq P$.
 
 Finally, denoting  $ C\!o\!f\colon\Mon\BV\xra{}\Bimon\BV$ the cofree bimonoid functor, that is the right adjoint of the forgetful functor $|-|\colon\Bimon\BV\xra{}\Mon\BV$, available in every locally presentable monoidal closed category as well (see \cite{HEP_QMI}), one obtains
$U = (|-|\circ U_{Bi})\dashv (P\circ C\!o\!f)$, such that $Lie \simeq P\circ C\!o\!f$ follows.
We note that we could not find the last equation in the literature. It says in particular that, for any $R$-algebra $\ma$,  the $R$-module $A$ is isomorphic to the $R$-module $Prim (C\!o\!f\ma)$ of primitive elements in the cofree bialgebra over $\ma$.
\end{rems}

For the convenience of the reader we visualize the various functors as follows. 
\begin{equation*}%\label{diag:}
\begin{aligned}
%\ar@{.>}[dl]^d
%\ar@/^2pc/@{->}[rr]^m
\xymatrix@=4em{
 \Lie\BV\ar@/^2pc/@{->}[rrr]^{U_{Bi}}\ar@<.5ex>[r]^R \ar@/^4pc/@{->}[rrrr]^U  &\Lie\BV_\Ab%\ar@/^2pc/@{.>}[rrr]^{U_\Ab} 
 \ar@<.5ex>[l]^I \ar@<.5ex>[r]^{U_H}
 &{\Hopf\BV }\ar@<.5ex>[r]^{E }\ar@<.5ex>[l]^{P_H}  &{\Bimon\BV }\ar@<.5ex>[l]^{C}\ar@/^2pc/@{.>}[lll]^P \ar@<.5ex>[r]^{|-| } &\Mon\BV \ar@/^4pc/@{->}[llll]^{Lie}\ar@<.5ex>[l]^{C\!o\!f}%\ar@/^2pc/@{.>}[lll]^{\Lambda_{ab}}
}
\end{aligned}
\end{equation*}

%%%%%%%%%%%%%%%%

\end{document}